\documentclass[10pt]{amsart}
\usepackage{amsrefs}
\usepackage{amssymb}
\usepackage[all]{xy}

\DeclareMathOperator{\Gal}{Gal}

\DeclareMathOperator{\Img}{Im}

\DeclareMathOperator{\Ker}{Ker}

\DeclareMathOperator{\Res}{Res}

\DeclareFontFamily{U}{wncy}{}
\DeclareFontShape{U}{wncy}{m}{n}{<->wncyr10}{}
\DeclareSymbolFont{mcy}{U}{wncy}{m}{n}
\DeclareMathSymbol{\Sha}{\mathord}{mcy}{"58}
\DeclareMathSymbol{\sha}{\mathord}{mcy}{"78}

\begin{document}

\newtheorem{thm}{Theorem}[section]
\newtheorem{cor}[thm]{Corollary}
\newtheorem{lem}[thm]{Lemma}
\newtheorem{prop}[thm]{Proposition}
\newtheorem{defin}[thm]{Definition}
\newtheorem{exam}[thm]{Example}
\newtheorem{examples}[thm]{Examples}
\newtheorem{rem}[thm]{Remark}
\newtheorem{case}{\sl Case}
\newtheorem{claim}{Claim}
\newtheorem{question}[thm]{Question}
\newtheorem{conj}[thm]{Conjecture}
\newtheorem*{notation}{Notation}
\swapnumbers
\newtheorem{rems}[thm]{Remarks}
\newtheorem*{acknowledgment}{Acknowledgment}
\newtheorem*{thmno}{Theorem}

\newtheorem{questions}[thm]{Questions}
\numberwithin{equation}{section}

\newcommand{\gr}{\mathrm{gr}}
\newcommand{\inv}{^{-1}}
\newcommand{\isom}{\cong}
\newcommand{\dbC}{\mathbb{C}}
\newcommand{\F}{\mathbb{F}}
\newcommand{\dbN}{\mathbb{N}}
\newcommand{\Q}{\mathbb{Q}}
\newcommand{\dbR}{\mathbb{R}}
\newcommand{\dbU}{\mathbb{U}}
\newcommand{\Z}{\mathbb{Z}}
\newcommand{\calG}{\mathcal{G}}
\newcommand{\K}{\mathbb{K}}

%%%%%%%%%%%%%%%%%
% New commands for our things - such as cocycles and so on

\newcommand{\hac}{\hat c}
\newcommand{\hatheta}{\hat\theta}
%%%%%%%%%%%%%%%%%%%%%%%%%%%%%%%%%%

\title[Galois-theoretic features for 1-smooth pro-$p$ groups]{Galois-theoretic features \\ for 1-smooth pro-$p$ groups}
\author{Claudio Quadrelli}
\address{Department of Mathematics and Applications, University of Milano Bicocca, 20125 Milan, Italy EU}
\email{claudio.quadrelli@unimib.it}
\date{\today}

\begin{abstract}
Let $p$ be a prime.
 A pro-$p$ group $G$ is said to be 1-smooth if it can be endowed with a continuous representation $\theta\colon G\to\mathrm{GL}_1(\Z_p)$ such that every open subgroup $H$ of $G$, together with the restriction $\theta\vert_H$, satisfies a formal version of Hilbert 90.
 We prove that every 1-smooth pro-$p$ group contains a unique maximal closed abelian normal subgroup, in analogy with a result by Engler and Koenigsmann on maximal pro-$p$ Galois groups of fields, and that if a 1-smooth pro-$p$ group is solvable, then it is locally uniformly powerful, in analogy with a result by Ware on maximal pro-$p$ Galois groups of fields.
 Finally we ask whether 1-smooth pro-$p$ groups satisfy a ``Tits' alternative''.
\end{abstract}

\subjclass[2010]{Primary 12G05; Secondary 20E18, 20J06, 12F10}

\keywords{Galois cohomology, Maximal pro-$p$ Galois groups, Bloch-Kato conjecture, Kummerian pro-$p$ pairs, Tits' alternative}

\maketitle
%%%%%%%%%%%%%%%%%%%%%%%%%%%%%%%%%%%%%%%%%%%%%%%%%

\section{Introduction}
\label{sec:intro}

Throughout the paper $p$ will denote a prime number, and $\K$ a field containing a root of unity of order $p$.
Let $\K(p)$ denote the compositum of all finite Galois $p$-extensions of $\K$.
The maximal pro-$p$ Galois group of $\K$, denoted by $G_{\K}(p)$, is the Galois group $\Gal(\K(p)/\K)$, and it coincides with the maximal pro-$p$ quotient of the absolute Galois group of ${\K}$.
Characterising maximal pro-$p$ Galois groups of fields among pro-$p$ groups is one of the most important --- and challenging --- problems in Galois theory.
One of the obstructions for the realization of a pro-$p$ group as maximal pro-$p$ Galois group for some field $\K$ is given by the Artin-Scherier theorem: the only finite group realizable as $G_{\K}(p)$ is the cyclic group of order 2 (cf. \cite{becker}).

The proof of the celebrated {\sl Bloch-Kato conjecture}, completed by M.~Rost and V.~Voevodsky with Ch.~Weibel's ``patch'' (cf. \cite{HW:book,voev,weibel})
provided new tools to study absolute Galois groups of field and their maximal pro-$p$ quotients (see, e.g., \cite{cem,cmq:fast,cq:bk,qw:cyc}).
In particular, the now-called Norm Residue Theorem implies that the {\sl $\Z/p$-cohomology algebra} of a maximal pro-$p$ Galois group $G_{\K}(p)$
\[
 H^\bullet(G_{\K}(p),\Z/p):=\bigoplus_{n\geq 0} H^n(G_{\K}(p),\Z/p),
\]
with $\Z/p$ a trivial $G_{\K}(p)$-module and endowed with the cup-product, is a quadratic algebra: 
i.e., all its elements of positive degree are combinations of products of elements of degree 1, and its defining relations are homogeneous relations of degree 2 (see \S~\ref{ssec:BK} below).
For instance, from this property one may recover the Artin-Schreier obstruction (see, e.g., \cite[\S~2]{cq:bk}).

More recently, a formal version of {\sl Hilbert 90} for pro-$p$ groups was employed 
to find further results on the structure of maximal pro-$p$ Galois groups (see \cite{eq:kummer,cq:noGal,qw:cyc}).
A pair $\calG=(G,\theta)$ consisting of a pro-$p$ group $G$ endowed with a continuous representation $\theta\colon G\to \mathrm{GL}_1(\Z_p)$ is called a {\sl pro-$p$ pair}.
For a pro-$p$ pair $\calG=(G,\theta)$ let $\Z_p(1)$ denote the continuous left $G$-module isomorphic to $\Z_p$ as an abelian pro-$p$ group, with $G$-action induced by $\theta$ (namely, $g.v=\theta(g)\cdot v$ for every $v\in\Z_p(1)$).
The pair $\calG$ is called a {\sl Kummerian pro-$p$ pair} if the canonical map 
\[
  H^1(G,\Z_p(1)/p^n)\longrightarrow H^1(G,\Z_p(1)/p),
\]
is surjective for every $n\geq1$.
Moreover the pair $\calG$ is said to be a {\sl 1-smooth} pro-$p$ pair if every closed subgroup $H$, endowed with the restriction $\theta\vert_H$, gives rise to a Kummerian pro-$p$ pair (see Definition~\ref{defin:kummer}).
By Kummer theory, the maximal pro-$p$ Galois group $G_{\K}(p)$ of a field $\K$, together with the pro-$p$ {\sl cyclotomic character} $\theta_{\K}\colon G_{\K}(p)\to\mathrm{GL}_1(\Z_p)$
(induced by the action of $G_{\K}(p)$ on the roots of unity of order a $p$-power lying in $\K(p)$)
gives rise to a 1-smooth pro-$p$ pair $\calG_{\K}$ (see Theorem~\ref{thm:galois}).

In \cite{dcf:lift} --- driven by the pursuit of an ``explicit'' proof
of the Bloch-Kato conjecture as an alternative to the proof by Voevodsky ---
 C.~De~Clerq and M.~Florence introduced the 1-smoothness property, and formulated the so-called ``Smoothness Conjecture'': namely, that it is possible to deduce the
surjectivity of the norm residue homomorphism (which is acknowledged to be the ``hard part'' of the Bloch-Kato conjecture) from the fact that 
$G_{\K}(p)$ together with the pro-$p$ cyclotomic character is a 1-smooth pro-$p$ pair
(see \cite[Conj.~14.25]{dcf:lift} and \cite[\S~3.1.6]{birs}, and Question~\ref{conj:smooth} below).

In view of the Smoothness Conjecture, it is natural to ask which properties of maximal pro-$p$ Galois groups of fields arise also for 1-smooth pro-$p$ pairs.
For example, the Artin-Scherier obstruction does: the only finite $p$-group which may complete into a 1-smooth pro-$p$ pair is the cyclic group $C_2$ of order 2, together with the non-trivial representation $\theta\colon C_2\to\{\pm1\}\subseteq\mathrm{GL}_1(\Z_2)$ (see Example~\ref{ex:Galois} below).

A pro-$p$ pair $\calG=(G,\theta)$ comes endowed with a distinguished closed subgroup: the {\sl $\theta$-center} $Z(\calG)$ of $\calG$, defined by
\[
 Z(\calG)=\left\langle\: h\in\Ker(\theta)\:\mid\: ghg^{-1}=h^{\theta(g)}\;\forall\:g\in G\:\right\rangle.
\]
This subgroup is abelian, and normal in $G$.
In \cite{EK}, A.~Engler and J.~Koenigsmann showed that if the maximal pro-$p$ Galois group $G_{\K}(p)$ of a field $\K$ is not cyclic then it has a unique maximal normal abelian closed subgroup (i.e., one containing all normal abelian closed subgroups of $G_{\K}(p)$), which coincides with the $\theta_{\K}$-center $Z(\calG_{\K})$, and the short exact sequence of pro-$p$ groups 
\[\xymatrix{\{1\}\ar[r] &  Z(\calG_{\K})\ar[r] &  G_{\K}(p)\ar[r] &  G_{\K}(p)/Z(\calG_{\K})\ar[r] & \{1\}}
\] splits.
We prove a group-theoretic analogue of Engler-Koenigsmann's result for 1-smooth pro-$p$ groups.

\begin{thm}\label{thm:Z intro}
 Let $G$ be a torsion-free pro-$p$ group, $G\not\simeq\Z_p$, endowed with a representation $\theta\colon G\to\mathrm{GL}_1(\Z_p)$ such that $\calG=(G,\theta)$ is a 1-smooth pro-$p$ pair.
 Then $Z(\calG)$ is the unique maximal normal abelian closed subgroup of $G$, and the quotient $G/Z(\calG)$ is a torsion-free pro-$p$ group.
\end{thm}

In \cite{ware}, R.~Ware proved the following result on maximal pro-$p$ Galois groups of fields:
if $G_{\K}(p)$ is solvable, then it is locally uniformly powerful, i.e., $G_{\K}(p)\simeq A\rtimes\Z_p$, where $A$ is a free abelian pro-$p$ group, and the right-side factor acts by scalar multiplication by a unit of $\Z_p$ (see \S~\ref{ssec:powerful}). 
%is torsion-free and for every finitely generated closed subgroup $H\subseteq G_{\K}(p)$ one has $[H,H]\subseteq H^p$ ($[H,H]\subseteq H^4$ if $p=2$), where $[H,H]$ denotes the closed commutator subgroup of $H$ and $H^p$ denotes the closed subgroup of $H$ generated by the $p$-th powers of the elements of $H$ (see Definition~\ref{defin:pwfl}).
We prove that the same property holds also for 1-smooth pro-$p$ groups.

\begin{thm}\label{thmA:intro}
 Let $G$ be a solvable torsion-free pro-$p$ group, endowed with a representation $\theta\colon G\to\mathrm{GL}_1(\Z_p)$ such that $\calG=(G,\theta)$ is a 1-smooth pro-$p$ pair.
 Then $G$ is locally uniformly powerful.
\end{thm}

This gives a complete description of solvable torsion-free pro-$p$ groups which may be completed into a 1-smooth pro-$p$ pair.
Moreover, Theorem~\ref{thmA:intro} settles the Smoothness Conjecture positively for the class of solvable pro-$p$ groups.

\begin{cor}\label{cor:intro}
 If $\calG=(G,\theta)$ is a 1-smooth pro-$p$ pair with $G$ solvable, then $G$ is a Bloch-Kato pro-$p$ group, i.e., the $\Z/p$-cohomology algebra of every closed subgroup of $G$ is quadratic. 
\end{cor}

\begin{rem}\label{rem:IlirSlobo}\rm
 Afther the submission of this paper, I.~Snopce and S.~Tanushevski showed in \cite{ilirslobo} that Theorems~\ref{thmA:intro}--\ref{thm:Z intro} hold for a wider class of pro-$p$ groups. 
 A pro-$p$ group is said to be {\sl Frattini-injective} if distinct finitely generated closed subgroups have distinct Frattini subgroups (cf. \cite[Def.~1.1]{ilirslobo}).
 By \cite[Thm.~1.11 and Cor.~4.3]{ilirslobo}, a pro-$p$ group which may complete into a 1-smooth pro-$p$ pair is Frattini-injective.
 By \cite[Thm.~1.4]{ilirslobo} a Frattini-injective pro-$p$ group has a unique maximal normal abelian closed subgroup, and by \cite[Thm.~1.3]{ilirslobo} a Frattini-injective pro-$p$ group is solvable if, and only if, it is locally uniformly powerful.
\end{rem}

A solvable pro-$p$ group does not contain a free non-abelian closed subgroup.
For Bloch-Kato pro-$p$ groups --- and thus in particular for maximal pro-$p$ Galois groups of fields containing a root of unity of order $p$ --- Ware proved the following Tits' alternative: 
either such a pro-$p$ group contains a free non-abelian closed subgroup; or it is locally uniformly powerful (see \cite[Cor.~1]{ware} and \cite[Thm.~B]{cq:bk}).
We conjecture that the same phenomenon occurs for 1-smooth pro-$p$ groups.

\begin{conj}
Let $G$ be a torsion-free pro-$p$ group which may be endowed with a representation $\theta\colon G\to\mathrm{GL}_1(\Z_p)$ such that $\calG=(G,\theta)$ is a 1-smooth pro-$p$ pair.
Then either $G$ is locally uniformly powerful, or $G$ contains a closed non-abelian free pro-$p$ group.
\end{conj}

%%%%%%%%%%%%%%%%%%%%%%%%%%%%%%%%%%%%%%%%%%%%%%%%%%%5
%%%%%%%%%%%55
%%%%%%%%%%%%%%%%%%%%%%%%%%%%%%%%%%%%%%%%%%%%%%
%%%%%%%%%%%%%%%%%%%%%%%%%%%%%%%%%%%%%%%%%%%%%%%%%%%5
%%%%%%%%%%%55
%%%%%%%%%%%%%%%%%%%%%%%%%%%%%%%%%%%%%%%%%%%%%%
%%%%%%%%%%%%%%%%%%%%%%%%%%%%%%%%%%%%%%%%%%%%%%

\section{Cyclotomic pro-$p$ pairs}

Henceforth, every subgroup of a pro-$p$ group will be tacitly assumed to be closed, and the generatos of a subgroup will be intended in the topological sense.

In particular, for a pro-$p$ group $G$ and a positive integer $n$, $G^{p^n}$ will denote the closed subgroup of $G$ generated by the $p^n$-th powers of all elements of $G$.
Moreover, for two elements $g,h\in G$, we set $$h^g=g^{-1}hg,\qquad\text{and}\qquad[h,g]=h^{-1}\cdot h^g,$$
and for two subgroups $H_1,H_2$ of $G$, $[H_1,H_2]$ will denote the closed subgroup of $G$ generated by all commutators $[h,g]$
with $h\in H_1$ and $g\in H_2$.
In particular, $G'$ will denote the commutator subgroup $[G,G]$ of $G$, and the Frattini subgroup $G^p\cdot G'$ of $G$ is denoted by $\Phi(G)$.
Finally, $d(G)$ will denote the minimal number of generatord of $G$, i.e., $d(G)=\dim(G/\Phi(G))$ as a $\Z/p$-vector space.

\subsection{Kummerian pro-$p$ pairs}

Let $1+p\Z_p=\{1+p\lambda\mid \lambda\in\Z_p\}\subseteq\mathrm{GL}_1(\Z_p)$ denote the pro-$p$ Sylow subgroup of the group of units of the ring of $p$-adic integers $\Z_p$.
A pair $\calG=(G,\theta)$ consisting of a pro-$p$ group $G$ and a continuous homomorphism $$\theta\colon G\longrightarrow1+p\Z_p$$ is called a {\sl cyclotomic pro-$p$ pair}, and the morphism $\theta$ is called an {\sl orientation} of $G$ (cf. \cite[\S~3]{ido:small} and \cite{qw:cyc}).

A cyclotomic pro-$p$ pair $\calG=(G,\theta)$ is said to be {\sl torsion-free} if $\Img(\theta)$ is torsion-free: this is the case if $p$ is odd; or if $p=2$ and $\Img(\theta)\subseteq 1+4\Z_2$. 
Observe that a cyclotomic pro-$p$ pair $\calG=(G,\theta)$ may be torsion-free even if $G$ has non-trivial torsion --- e.g., if $G$ is the cyclic group of order $p$ and $\theta$ is constantly equal to 1.
Given a cyclotomic pro-$p$ pair $\calG=(G,\theta)$ one has the following constructions:
\begin{itemize}
 \item[(a)] if $H$ is a subgroup of $G$, $\Res_H(\calG)=(H,\theta\vert_H)$;
 \item[(b)] if $N$ is a normal subgroup of $G$ contained in $\Ker(\theta)$, then $\theta$ induces an orientation $\bar \theta\colon G/N\to1+p\Z_p$, and we set $\calG/N=(G/N,\bar\theta)$;
 \item[(c)] if $A$ is an abelian pro-$p$ group, we set $A\rtimes\calG=(A\rtimes G,\theta\circ\pi)$, with $a^g=a^{\theta(g)^{-1}}$ for all $a\in A$, $g\in G$, and $\pi$ the canonical projection $A\rtimes G\to G$.
\end{itemize}

Given a cyclotomic pro-$p$ pair $\calG=(G,\theta)$, the pro-$p$ group $G$ has two distinguished subgroups:
\begin{itemize}
 \item[(a)] the subgroup 
\begin{equation}\label{eq:K subgroup}
   K(\calG)=\left\langle\left. h^{-\theta(g)}\cdot h^{g^{-1}}\right|g\in G,h\in\Ker(\theta)\right\rangle
\end{equation}
introduced in \cite[\S~3]{eq:kummer};
\item[(b)] the $\theta$-center
\begin{equation}
  Z(\calG)=\left\langle h\in\Ker(\theta)\left|ghg^{-1}=h^{\theta(g)}\;\forall\:g\in G\right.\right\rangle
\end{equation}
introduced in \cite[\S~1]{cq:bk}.
\end{itemize}
Both $Z(\calG)$ and $K(\calG)$ are normal subgroups of $G$, and they are contained in $\Ker(\theta)$.
Moreover, $Z(\calG)$ is abelian, while
\[
   K(\calG)\supseteq\Ker(\theta)',\qquad\text{and} \qquad K(\calG)\subseteq\Phi(G).
  \]
Thus, the quotient $\Ker(\theta)/K(\calG)$ is abelian, and if $\calG$ is torsion-free one has an isomorphism of pro-$p$ pairs 
\begin{equation}
 \calG/K(\calG)\simeq(\Ker(\theta)/K(\calG))\rtimes (\calG/\Ker(\theta)),
\end{equation}
namely, $G/K(\calG)\simeq(\Ker(\theta)/K(\calG))\rtimes (G/\Ker(\theta))$ (where the action is induced by $\theta$, in the latter), and both pro-$p$ groups are endowed with the orientation induced by $\theta$ (cf. \cite[Prop.~3.1]{cq:1smooth}).

\begin{defin}\label{defin:kummer}\rm
Given a cyclotomic pro-$p$ pair $\calG=(G,\theta)$, let $\Z_p(1)$ denote the continuous $G$-module of rank 1 induced by $\theta$, i.e., $\Z_p(1)\simeq\Z_p$ as abelian pro-$p$ groups, and $g.\lambda=\theta(g)\cdot\lambda$ for every $\lambda\in\Z_p(1)$.
The pair $\calG$ is said to be {\sl Kummerian} if for every $n\geq1$ the map
\begin{equation}\label{eq:epimorphism}
  H^1(G,\Z_p(1)/p^n)\longrightarrow H^1(G,\Z_p(1)/p),
\end{equation}
induced by the epimorphism of $G$-modules $\Z_p(1)/p^n\to\Z_p(1)/p$, is surjective.
Moreover, $\calG$ is {\sl 1-smooth} if $\Res_H(\calG)$ is Kummerian for every subgroup $H\subseteq G$.
\end{defin}

Observe that the action of $G$ on $\Z_p(1)/p$ is trivial, as $\Img(\theta)\subseteq1+p\Z_p$.
We say that a pro-$p$ group $G$ may complete into a Kummerian, or 1-smooth, pro-$p$ pair if there exists an orientation $\theta\colon G\to1+p\Z_p$ such that the pair $(G,\theta)$ is Kummerian, or 1-smooth.

Kummerian pro-$p$ pairs and 1-smooth pro-$p$ pairs were introduced in \cite{eq:kummer} and in \cite[\S~14]{dcf:lift} respectively.
In \cite{qw:cyc}, if $\calG=(G,\theta)$ is a 1-smooth pro-$p$ pair, the orientation $\theta$ is said to be {\sl 1-cyclotomic}.
Note that in \cite[\S~14.1]{dcf:lift}, a pro-$p$ pair is defined to be 1-smooth if the maps \eqref{eq:epimorphism} are surjective for every {\sl open} subgroup of $G$, yet by a limit argument this implies also that the maps \eqref{eq:epimorphism} are surjective also for every {\sl closed} subgroup of $G$ (cf. \cite[Cor.~3.2]{qw:cyc}). 

\begin{rem}\label{rem:Zp1}\rm
 Let $\calG=(G,\theta)$ be a cyclotomic pro-$p$ pair. 
 Then $\calG$ is Kummerian if, and only if, the map 
 \[
  H^1_{\mathrm{cts}}(G,\Z_p(1))\longrightarrow H^1(G,\Z_p(1)/p),
 \]
induced by the epimorphism of continuous left $G$-modules $\Z_p(1)\twoheadrightarrow\Z_p(1)/p$, is surjective (cf. \cite[Prop.~2.1]{qw:cyc}) --- here $H_{\mathrm{cts}}^*$ denotes continuous cochain cohomology as introduced by J.~Tate in \cite{tate}.
\end{rem}

One has the following group-theoretic characterization of Kummerian torsion-free pro-$p$ pairs (cf. \cite[Thm.~5.6 and Thm.~7.1]{eq:kummer} and \cite[Thm.~1.2]{cq:chase}).

\begin{prop}\label{prop:kummer}
 A torsion-free cyclotomic pro-$p$ pair $\calG=(G,\theta)$ is Kummerian if and only if $\Ker(\theta)/K(\calG)$ is a free abelian pro-$p$ group.
\end{prop}

\begin{rem}\label{rem:abs torfree}\rm
Let $\calG=(G,\theta)$ be a cyclotomic pro-$p$ pair with $\theta\equiv\mathbf{1}$, i.e., $\theta$ is constantly equal to 1.
Since $K(\calG)=G'$ in this case, $\calG$ is Kummerian if and only if the quotient $G/G'$ is torsion-free. 
Hence, by Proposition~\ref{prop:kummer}, $\calG$ is 1-smooth if and only if $H/H'$ is torsion-free for every subgroup $H\subseteq G$.
Pro-$p$ groups with such property are called {\sl absolutely torsion-free}, and they were introduced by T.~W\"urfel in \cite{wurf}.
In particular, if $\calG=(G,\theta)$ is a 1-smooth pro-$p$ pair (with $\theta$ non-trivial), then $\Res_{\Ker(\theta)}(\calG)=(\Ker(\theta),\mathbf{1})$ is again 1-smooth, and thus $\Ker(\theta)$ is absolutely torsion-free.
Hence, a pro-$p$ group which may complete into a 1-smooth pro-$p$ pair is an absolutely torsion-free-by-cyclic pro-$p$ group.
\end{rem}

\begin{exam}\label{ex:1-smooth}\rm  
\begin{itemize}
\item[(a)] A cyclotomic pro-$p$ pair $(G,\theta)$ with $G$ a free pro-$p$ group is 1-smooth for any orientation $\theta\colon G\to1+p\Z_p$ (cf. \cite[\S~2.2]{qw:cyc}).
\item[(b)] A cyclotomic pro-$p$ pair $(G,\theta)$ with $G$ an infinite Demushkin pro-$p$ group is 1-smooth if and only if $\theta\colon G\to1+p\Z_p$ is defined as in \cite[Thm.~4]{labute:demushkin} (cf. \cite[Thm.~7.6]{eq:kummer}).
E.g., if $G$ has a minimal presentation
\[
 G=\left\langle\:x_1,\ldots,x_d\:\mid\:x_1^{p^f}[x_1,x_2]\cdots[x_{d-1},x_d]=1\:\right\rangle
\]
with $f\geq1$ (and $f\geq2$ if $p=2$), then $\theta(x_2)=(1-p^f)^{-1}$, while $\theta(x_i)=1$ for $i\neq2$.
\item[(c)] For $p\neq2$ let $G$ be the pro-$p$ group with minimal presentation 
$$ G=\langle x,y,z\mid  [x,y]=z^p\rangle.$$
Then the pro-$p$ pair $(G,\theta)$ is not Kummerian for any orientation $\theta\colon G\to1+p\Z_p$ (cf. \cite[Thm.~8.1]{eq:kummer}).
\item[(d)] Let 
$$H=\left\{\left(\begin{array}{ccc} 1 & a & c \\ 0 & 1 & b \\ 0 & 0 &1 \end{array}\right)\mid a,b,c\in\Z_p\right\}$$
 be the Heisenberg pro-$p$ group.
The pair $(H,\mathbf{1})$ is Kummerian, as $H/H'\simeq\Z_p^2$, but $H$ is not absolutely torsion-free.
In particular, $H$ can not complete into a 1-smooth pro-$p$ pair (cf. \cite[Ex.~5.4]{cq:1smooth}).
\item[(e)] The only 1-smooth pro-$p$ pair $(G,\theta)$ with $G$ a finite $p$-group is the cyclic group of order 2 $G\simeq\Z/2$, endowed with the only non-trivial orientation $\theta\colon G\twoheadrightarrow\{\pm1\}\subseteq 1+2\Z_2$ (cf. \cite[Ex.~3.5]{eq:kummer}).
\end{itemize}
\end{exam}

\begin{rem}\label{rem:torfree}\rm
By Example~\ref{ex:1-smooth}--(e), if $\calG=(G,\theta)$ is a torsion-free 1-smooth pro-$p$ pair, then $G$ is torsion-free.
\end{rem}

A torsion-free pro-$p$ pair $\calG=(G,\theta)$ is said to be {\sl $\theta$-abelian} if the following equivalent conditions hold:
\begin{itemize}
 \item[(i)] $\Ker(\theta)$ is a free abelian pro-$p$ group, and $\calG\simeq \Ker(\theta)\rtimes(\calG/\Ker(\theta))$;
 \item[(ii)] $Z(\calG)$ is a free abelian pro-$p$ group, and $Z(\calG)=\Ker(\theta)$;
 \item[(iii)] $\calG$ is Kummerian and $K(\calG)=\{1\}$
\end{itemize}
(cf. \cite[Prop.~3.4]{cq:bk} and \cite[\S~2.3]{cq:chase}).
Explicitly, a torsion-free pro-$p$ pair $\calG=(G,\theta)$ is $\theta$-abelian if and only if $G$ has a minimal presentation
\begin{equation}\label{eq:pres thetabelian}
  G=\left\langle x_0,x_i,i\in I\mid [x_0,x_i]=x_i^{q},[x_i,x_j]=1\:\forall\:i,j\in I\right\rangle\simeq\Z_p^I\rtimes \Z_p
\end{equation}
for some set $I$ and some $p$-power $q$ (possibly $q=p^\infty=0$), and in this case $\Img(\theta)=1+q\Z_p$.
In particular, a $\theta$-abelian pro-$p$ pair is also 1-smooth, as every open subgroup $U$ of $G$ is again isomorphic to $\Z_p^I\rtimes\Z_p$, with action induced by $\theta\vert_U$, and therefore $\Res_U(\calG)$ is $\theta\vert_U$-abelian.

\begin{rem}\label{rem:fratcov}\rm
From \cite[Thm.~5.6]{eq:kummer}, one may deduce also the following group-theoretic characterization of Kummerian pro-$p$ pairs: a pro-$p$ group $G$ may complete into a Kummerian oriented pro-$p$ group if, and only if, there exists an epimorphism of pro-$p$ groups $\varphi\colon G\twoheadrightarrow\bar G$ such that $\bar G$ has a minimal presentation \eqref{eq:pres thetabelian}, and $\Ker(\varphi)$ is contained in the Frattini subgroup of $G$ (cf., e.g., \cite[Prop.~3.11]{qw:bogo}).
\end{rem}

\begin{rem}\label{rem:unique}\rm
 If $G\simeq \Z_p$, then the pair $(G,\theta)$ is $\theta$-abelian, and thus also 1-smooth, for any orientation $\theta\colon G\to1+p\Z_p$.
 
 On the other hand, if $\calG=(G,\theta)$ is a $\theta$-abelian pro-$p$ pair with $d(G)\geq2$, then $\theta$ is the only orientation which may complete $G$ into a 1-smooth pro-$p$ pair.
 Indeed, let $\calG'=(G,\theta')$ be a cyclotomic pro-$p$ pair, with $\theta'\colon G\to1+p\Z_p$ different to $\theta$, and let $\{x_0,x_i,i\in I\}$ be a minimal generating set of $G$ as in the presentation \eqref{eq:pres thetabelian} --- thus, $\theta(x_i)=1$ for all $i\in I$, and $\theta(x_0)\in1+q\Z_p$.
 Then for some $i\in I$ one has $\theta'\vert_H\not\equiv\theta\vert_H$, with $H$ the subgroup of $G$ generated by the two elements $x_0$ and $x_i$.
 In particular, one has $\theta([x_0,x_i])=\theta'([x_0,x_i])=1$.

 Suppose that $\calG'$ is 1-smooth.
 If $\theta'(x_i)\neq1$, then 
 \[
  x_i^q=x_i\cdot x_i^q\cdot x_i^{-1}=(x_i^q)^{\theta'(x_i)}=x_i^{q\theta'(x_i)},
 \]
hence $x_i^{q(1-\theta'(x_i))}=1$, a contradiction as $G$ is torsion-free by Remark~\ref{rem:torfree}.
If $\theta'(x_i)=1$ then necessarily $\theta'(x_0)\neq\theta(x_0)$, and thus 
\[
 x_i^{\theta(x_0)}=x_0\cdot x_i\cdot x_0^{-1}=x_i^{\theta'(x_0)},
\]
hence $x_i^{\theta(x_0)-\theta'(x_0)}=1$, again a contradiction as $G$ is torsion-free.
(See also \cite[Cor.~3.4]{qw:cyc}.)
\end{rem}

%%%%%%%%%%%%%%%%%%%%%%%%%%%%%55

\subsection{The Galois case}\label{ssec:Galois}

Let $\K$ be a field containing a root of 1 of order $p$, and let $\mu_{p^\infty}$ denote the group of roots of 1 of order a $p$-power contained in the separable closure of $\K$.
Then $\mu_{p^\infty}\subseteq \K(p)$, and the action of the maximal pro-$p$ Galois group $G_{\K}(p)=\Gal(\K(p)/\K)$ on $\mu_{p^\infty}$ induces a continuous homomorphism $$\theta_{\K}\colon G_{\K}(p)\longrightarrow1+p\Z_p$$ --- called the {\sl pro-$p$ cyclotomic character of $G_{\K}(p)$} ---, as the group of the automorphisms of $\mu_{p^{\infty}}$ which fix the roots of order $p$ is isomorphic to $1+p\Z_p$ (see, e.g., \cite[p.~202]{ido:book} and \cite[\S~4]{eq:kummer}).
In particular, if $\K$ contains a root of 1 of order $p^k$ for $k\geq1$, then $\Img(\theta_{\K})\subseteq 1+p^k\Z_p$.

Set $\calG_{\K}=(G_{\K}(p),\theta_{\K})$. 
Then by Kummer theory one has the following (see, e.g., \cite[Thm.~4.2]{eq:kummer}).

\begin{thm}\label{thm:galois}
 Let $\K$ be a field containing a root of 1 of order $p$.
 Then $\calG_{\K}=(G_{\K}(p),\theta_{\K})$ is 1-smooth.
\end{thm}

1-smooth pro-$p$ pairs share the following properties with maximal pro-$p$ Galois groups of fields.

\begin{exam}\label{ex:Galois}\rm
 \begin{itemize}
  \item[(a)] The only finite $p$-group which occurs as maximal pro-$p$ Galois group for some field $\K$ 
  is the cyclic group of order 2, and this follows from the pro-$p$ version of the Artin-Schreier Theorem (cf.~\cite{becker}).
  Likewise, the only finite $p$-group which may complete into a 1-smooth pro-$p$ pair, is the cyclic group of order 2
  (endowed with the only non-trivial orientation onto $\{\pm1\}$), as it follows from Example~\ref{ex:1-smooth}--(e) and Remark~\ref{rem:torfree}.
  \item[(b)] If $x$ is an element of $G_{\K}(2)$ for some field $\K$ and $x$ has order 2, then $x$ self-centralizes (cf. \cite[Prop.~2.3]{craven}).
  Likewise, if $x$ is an element of a pro-$2$ group $G$ which may complete into a 1-smooth pro-2 pair, then $x$ self-centralizes (cf. \cite[\S~6.1]{qw:cyc}).
  \end{itemize}
\end{exam}

%%%%%%%%%%%%%%%%%%%%%%%%%%%%%%%%%%%%%%%%%%%%%%55

\subsection{Bloch-Kato and the Smoothness Conjecture}\label{ssec:BK}

A non-negatively graded algebra $A_\bullet=\bigoplus_{n\geq0}A_n$ over a field $\F$, with $A_0=\F$,
is called a {\sl quadratic algebra} if it is 1-generated --- i.e., every element is a combination of products
of elements of degree 1 ---, and its relations are generated by homogeneous relations of degree 2.
One has the following definitions (cf. \cite[Def.~14.21]{dcf:lift} and \cite[\S~1]{cq:bk}).

\begin{defin}\label{defin:BK}\rm
\rm Let $G$ be a pro-$p$ group, and let $n\geq1$.
Cohomology classes in the image of the natural cup-product
\[
 H^1(G,\Z/p)\times\ldots\times H^1(G,\Z/p)\overset{\cup}{\longrightarrow} H^n(G,\Z/p)
\]
are called {\sl symbols} (relative to $\Z/p$, wieved as trivial $G$-module).
\begin{itemize}
 \item[(i)] If for every open subgroup $U\subseteq G$ every element $\alpha\in H^n(U,\Z/p)$, for every $n\geq1$,
can be written as $$\alpha=\mathrm{cor}_{V_1,U}^n(\alpha_1)+\ldots+\mathrm{cor}^n_{V_r,U}(\alpha_r),$$ with $r\geq1$,
where $\alpha_i\in H^n(V_i,\Z/p)$ is a symbol and $$\mathrm{cor}_{V_i,U}^n\colon H^n(V_i,\Z/p)\longrightarrow H^n(U,\Z/p)$$
is the {\sl corestriction map} (cf. \cite[Ch.~I, \S~5]{nsw:cohn}), for some open subgroups $V_i\subseteq U$,
then $G$ is called a {\sl weakly Bloch-Kato pro-$p$ group}.
 \item[(ii)] If for every closed subgroup $H\subseteq G$ the $\Z/p$-cohomology algebra $$H^\bullet(H,\Z/p)=\bigoplus_{n\geq0}H^n(H,\Z/p),$$
endowed with the cup-product, is a quadratic algebra over $\Z/p$,
then $G$ is called a {\sl Bloch-Kato pro-$p$ group}.
As the name suggests, a Bloch-Kato pro-$p$ group is also weakly Bloch-Kato.
\end{itemize}
\end{defin}

By the Norm Residue Theorem, if $\K$ contains a root of unity of order $p$,
 then the maximal pro-$p$ Galois group $G_{\K}(p)$ is Bloch-Kato.
The pro-$p$ version of the ``Smoothness Conjecture'', formulated by De~Clerq and Florence, states that being 1-smooth
is a sufficient condition for a pro-$p$ group to be weakly Bloch-Kato
(cf. \cite[Conj.~14.25]{dcf:lift}).
 
\begin{conj}\label{conj:smooth}
 Let $\calG=(G,\theta)$ be a 1-smooth pro-$p$ pair.
 Then $G$ is weakly Bloch-Kato.
\end{conj}

In the case of $\calG=\calG_{\K}$ for some field $\K$ containing a root of 1 of order $p$,
using Milnor $K$-theory one may show that the weak Bloch-Kato
condition implies that $H^\bullet(G,\Z/p)$ is 1-generated (cf. \cite[Rem.~14.26]{dcf:lift}).
In view of Theorem~\ref{thm:galois}, a positive answer to the Smoothness Conjecture would provide 
a new proof of the surjectivity of the norm residue isomorphism, i.e., the ``surjectivity'' half of the Bloch-Kato conjecture
(cf. \cite[\S~1.1]{dcf:lift}).

Conjecture~\ref{conj:smooth} has been settled positively for the following classes of pro-$p$ groups.
\begin{itemize}
 \item[(a)] Finite $p$-groups: indeed, if $\calG=(G,\theta)$ is a 1-smooth pro-$p$ pair with $G$ a finite (non-trivial) $p$-group, then by Example~\ref{ex:1-smooth}--(e) $p=2$, $G$ is a cyclic group of order two and $\theta\colon G\twoheadrightarrow\{\pm1\}$, so that $\calG\simeq(\Gal(\dbC/\dbR),\theta_{\dbR})$, and $G$ is Bloch-Kato.
 \item[(b)] Analytic pro-$p$ groups: indeed if $\calG=(G,\theta)$ is a 1-smooth pro-$p$ pair with $G$ a $p$-adic analytic pro-$p$ group, then by \cite[Thm.~1.1]{cq:1smooth} $G$ is locally uniformly powerful and thus Bloch-Kato (see \S~\ref{ssec:powerful} below).
 \item[(c)] Pro-$p$ completions of right-angled Artin groups: indeed, in \cite{SZ:RAAGs} it is shown that if $\calG=(G,\theta)$ is a 1-smooth pro-$p$ pair with $G$ the pro-$p$ completion of a right-angled Artin group induced by a simplicial graph $\Gamma$, then necessarily $\theta$ is trivial and $\Gamma$ has the diagonal property --- namely, $G$ may be constructed starting from free pro-$p$ groups by iterating the following two operations: free pro-$p$ products, and direct products with $\Z_p$ ---, and thus $G$ is Bloch-Kato (cf. \cite[Thm.~1.2]{SZ:RAAGs}).
\end{itemize}

%%%%%%%%%%%%%%%%%%%%%%%%%%%%%%%%%%%%%%%%%%%%%%%%%%%%%%%55
%%%%%%%%%%%555
%%%%%%%%%%%%%%%%%%%%%%%%%%%%%%%%%%%%%%%%%%%%%%%%%%%%%%%%%5

\section{Normal abelian subgroups}\label{sec:normabelian}

\subsection{Powerful pro-$p$ groups}\label{ssec:powerful}

\begin{defin}\label{defin:pwfl}\rm
 A finitely generated pro-$p$ group $G$ is said to be {\sl powerful} if one has $G'\subseteq G^p$, and also $G'\subseteq G^4$ if $p=2$.
 A powerful pro-$p$ group which is also torsion-free and finitely generated is called a {\sl uniformly powerful} pro-$p$ group.
% A pro-$p$ group $G$ is said to be {\sl locally uniformly powerful} if every finitely generated closed subgroup of $G$ is uniform.
\end{defin}

For the properties of powerful and uniformly powerful pro-$p$ groups we refer to \cite[Ch.~4]{ddsms}.

A pro-$p$ group whose finitely generated subgroups are uniformly powerful, is said to be {\sl locally uniformly powerful}. 
As mentioned in the Introduction, a pro-$p$ group $G$ is locally uniformly powerful if, and only if, $G$ has a minimal presentation \eqref{eq:pres thetabelian} --- i.e., $G$ is locally powerful if, and only if, there exists an orientation $\theta\colon G\to1+p\Z_p$ such that $(G,\theta)$ is a torsion-free $\theta$-abelian pro-$p$ pair (cf. \cite[Thm.~A]{cq:bk} and \cite[Prop.~3.5]{cmq:fast}).

Therefore, a locally uniformly powerful pro-$p$ group $G$ comes endowed authomatically with an orientation $\theta\colon G\to1+p\Z_p$ such that $\calG=(G,\theta)$ is a 1-smooth pro-$p$ pair.
In fact, finitely generated locally uniformly powerful pro-$p$ groups are precisely those uniformly powerful pro-$p$ groups which may complete into a 1-smooth pro-$p$ pair (cf. \cite[Prop.~4.3]{cq:1smooth}).

\begin{prop}\label{prop:pwfl-1smooth}
 Let $\calG=(G,\theta)$ be a 1-smooth torsion-free pro-$p$ pair. 
 If $G$ is locally powerful, then $\calG$ is $\theta$-abelian, and thus $G$ is locally uniformly powerful.
\end{prop}

It is well-known that the $\Z/p$-cohomology algebra of a pro-$p$ group $G$ with minimal presentation \eqref{eq:pres thetabelian} is the exterior $\Z/p$-algebra 
\[
  H^\bullet(H,\Z/p)\simeq \bigwedge_{n\geq0} H^1(H,\Z/p)
\]
--- if $p=2$ then $\bigwedge_{n\geq0} V$ is defined to be the quotient of the tensor algebra over $\Z/p$ generated by $V$ by the two-sided ideal generated by the elements $v\otimes v$, $v\in V$ ---, so that $H^\bullet(G,\Z/p)$ is quadratic.
Moreover, every subgroup $H\subseteq G$ is again locally uniformly powerful, and thus also $H^\bullet(H,\Z/p)$ is quadratic.
Hence, a locally uniformly powerful pro-$p$ group is Bloch-Kato.

%%%%%%%%%%%%%%%%%%%%%%%%%%%%%%%%%%%%%%%%%%%%%%%%%%%%%%%%%%%%%%%%%%%%%%%%%

\subsection{Normal abelian subgroups of maximal pro-$p$ Galois groups}

Let $\K$ be a field containing a root of 1 of order $p$ (and also $\sqrt{-1}$ if $p=2$).
In Galois theory one has the following result, due to A.~Engler, J.~Koenigsmann and J.~Nogueira (cf. \cite{EN} and \cite{EK}).

\begin{thm}\label{thm:EK}
 Let $\K$ be a field containing a root of 1 of order $p$ (and also $\sqrt{-1}$ if $p=2$), and suppose that the maximal pro-$p$ Galois group $G_{\K}(p)$ of $\K$ is not isomorphic to $\Z_p$.
 Then $G_{\K}(p)$ contains a unique maximal abelian normal subgroup.
 \end{thm}

By \cite[Thm.~7.7]{qw:cyc}, such a maximal abelian normal subgroup coincides with the $\theta_{\K}$-center $Z(\calG_{\K})$ of the pro-$p$ pair $\calG_{\K}=(G_{\K}(p),\theta_{\K})$ induced by the pro-$p$ cyclotomic character $\theta_{\K}$ (cf. \S~\ref{ssec:Galois}).
Moreover, the field $\K$ admits a {\sl $p$-Henselian valuation} with residue characteristic not $p$ and non-$p$-divisible value group, such that the {\sl residue field} $\kappa$ of such a valuation gives rise to the cyclotomic pro-$p$ pair $\calG_{\kappa}$ isomorphic to $\calG_{\K}/Z(\calG_{\K})$, and the induced short exact sequence of pro-$p$ groups
\begin{equation}\label{eq:transition}
   \xymatrix{ \{1\}\ar[r] & Z(\calG_{\K})\ar[r] &G_{\K}(p)\ar[r] & G_{\kappa}(p)\ar[r] &\{1\}}
\end{equation}
splits (cf. \cite[\S~1]{EK} and \cite[Ex.~22.1.6]{ido:book} --- for the definitions related to $p$-henselian valuations of fields we direct the reader to \cite[\S~15.3]{ido:book}).
In particular, $G_{\K}(p)/Z(\calG_{\K})$ is torsion-free.

\begin{rem}\label{rem:split}\rm
By \cite[Thm.~1.2 and Thm.~7.7]{qw:cyc}, Theorem~\ref{thm:EK} and the splitting of \eqref{eq:transition} generalize to 1-smooth pro-$p$ pairs whose underlying pro-$p$ group is Bloch-Kato.
Namely, if $\calG=(G,\theta)$ is a 1-smooth pro-$p$ pair with $G$ a Bloch-Kato pro-$p$ group, then $Z(\calG)$ is the unique maximal abelian normal subgroup of $G$, and it has a complement in $G$.
\end{rem}

%%%%%%%%%%%%%%%%%%%%%%%%%%%%%%%%%%%%%%%%%%%%%%%

\subsection{Proof of Theorem~\ref{thm:Z intro}}

In order to prove Theorem~\ref{thm:Z intro} (and also Theorem~\ref{thmA:intro} later on), we need the following result.

\begin{prop}\label{prop:prop}
 Let $\calG=(G,\theta)$ be a torsion-free 1-smooth pro-$p$ pair, with $d(G)=2$ and $G=\langle x,y\rangle$.
 If $[[x,y],y]=1$, then $\Ker(\theta)=\langle y\rangle$ and $$xyx^{-1}=y^{\theta(x)}.$$
\end{prop}

\begin{proof}
 Let $H$ be the subgroup of $G$ generated by $y$ and $[x,y]$.
 Recall that by Remark~\ref{rem:torfree}, $G$ (and hence also $H$) is torsion-free.
 
 If $d(H)=1$ then $H\simeq \Z_p$, as $H$ is torsion-free.
 Moreover, $H$ is generated by $y$ and $x^{-1}yx$, and thus $xHx^{-1}\subseteq H$.
Therefore, $x$ acts on $H\simeq\Z_p$ by multiplication by $1+p\lambda$ for some $\lambda\in \Z_p$.
If $\lambda=0$ then $G$ is abelian, and thus $G\simeq \Z_p^2$ as it is absolutely torsion-free, and $\theta\equiv\mathbf{1}$ by Remark~\ref{rem:unique}.
If $\lambda\neq0$ then $x$ acts non-trivially on the elements of $H$, and thus $\langle x\rangle\cap H=\{1\}$ and $G=H\rtimes\langle x\rangle$: by \eqref{eq:pres thetabelian}, $(G,\theta')$ is a $\theta'$-abelian pro-$p$ pair, with $\theta'\colon G\to1+p\Z_p$ defined by $\theta'(x)=1+p\lambda$ and $\theta'(y)=1$.
 By Remark~\ref{rem:unique}, one has $\theta'\equiv\theta$, and thus $\theta(x)=1+p\lambda$ and $\theta(y)=1$.
 
 If $d(H)=2$, then $H$ is abelian by hypothesis, and torsion-free, and thus $(H,\theta')$ is $\theta'$-abelian, with $\theta'\equiv{\mathbf{1}}\colon H\to1+p\Z_p$ trivial.
 By Remark~\ref{rem:unique}, one has $\theta'=\theta\vert_H$, and thus $y,[x,y]\in\Ker(\theta)$.
Now put $z=[x,y]$ and $t=y^p$, and let $U$ be the open subgroup of $G$ generated by $x,z,t$.
Clearly, $\Res_U(\calG)$ is again 1-smooth.
By hypothesis one has $z^y=z$, and hence commutator calculus yields
\begin{equation}\label{eq:1proof}
 [x,t]=[x,y^p]=z\cdot z^y\cdots z^{y^{p-1}}=z^p.
\end{equation}

Put $\lambda=1-\theta(x)^{-1}\in p\Z_p$.
Since $t\in\Ker(\theta)$, by \eqref{eq:K subgroup} $[x,t]\cdot t^{-\lambda}$ lies in $K(\Res_U(\calG))$.
Since $t$ and $z$ commute, from \eqref{eq:1proof} one deduces 
\begin{equation}\label{eq:2proof}
 [x,t]t^{-\lambda}=z^pt^{-\lambda}=z^pt^{-\frac{\lambda}{p}p}=\left(zt^{-\lambda/p}\right)^p\in K(\Res_U(\calG)).
\end{equation}
Moreover, $zt^{-\lambda/p}\in\Ker(\theta\vert_U)$.
Since $\Res_U(\calG)$ is 1-smooth, by Proposition~\ref{prop:kummer} the quotient $\Ker(\theta\vert_U)/K(\Res_U(\calG))$ is a free abelian pro-$p$ group, and therefore \eqref{eq:2proof} implies that also $zt^{-\lambda/p}$ is an element of $K(\Res_U(\calG))$.

Since $K(\Res_U(\calG))\subseteq \Phi(U)$, one has $z\equiv t^{\lambda/p}\bmod\Phi(U)$.
Then by \cite[Prop.~1.9]{ddsms} $d(U)=2$ and $U$ is generated by $x$ and $t$.
Since $[x,t]\in U^p$ by \eqref{eq:1proof}, the pro-$p$ group $U$ is powerful.
Therefore, $\Res_U(\calG)$ is $\theta\vert_U$-abelian by Proposition~\ref{prop:pwfl-1smooth}.
In particular, the subgroup $K(\Res_U(\calG))$ is trivial, and thus $$[x,y]=z=t^{\lambda/p}=y^{1-\theta(x)^{-1}},$$
and the claim follows.
\end{proof}

Proposition~\ref{prop:prop} is a generalization of \cite[Prop.~5.6]{cq:1smooth}.

\begin{thm}\label{thm:Z sec}
 Let $\calG=(G,\theta)$ be a torsion-free 1-smooth pro-$p$ pair, with $d(G)\geq2$. 
\begin{itemize}
 \item[(i)] The $\theta$-center $Z(\calG)$ is the unique maximal abelian normal subgroup of $G$.
 \item[(ii)] The quotient $G/Z(\calG)$ is a torsion-free pro-$p$ group.
\end{itemize} 
\end{thm}

\begin{proof}
Recall that $G$ is torsion-free by Remark~\ref{rem:torfree}.
 Since $Z(\calG)$ is an abelian normal subgroup of $G$ by definition, in order to prove (i) we need to show that if $A$ is an abelian normal subgroup of $G$, then $A\subseteq Z(\calG)$.

 First, we show that $A\subseteq \Ker(\theta)$.
 If $A\simeq\Z_p$, let $y$ be a generator of $A$.
 For every $x\in G$ one has $xyx^{-1}\in A$, and thus $xyx^{-1}=y^{\lambda}$, for some $\lambda\in1+p\Z_p$.
 Let $H$ be the subgroup of $G$ generated by $x$ and $y$, for some $x\in G$ such that $d(H)=2$.
 Then the pair $(H,\theta')$ is $\theta'$-abelian for some orientation $\theta'\colon H\to1+p\Z_p$ such that $y\in\Ker(\theta')$, as $H$ has a presentation as in \eqref{eq:pres thetabelian}.
 Since both $\Res_H(\calG)$ and $(H,\theta')$ are 1-smooth pro-$p$ pairs, by Remark~\ref{rem:unique} one has $\theta'=\theta\vert_H$, and thus $A\subseteq\Ker(\theta)$.
 
 If $A\not\simeq\Z_p$, then $A$ is a free abelian pro-$p$ group with $d(A)\geq2$, as $G$ is torsion-free.
 Therefore, by Remark~\ref{rem:abs torfree} the pro-$p$ pair $(A,\mathbf{1})$ is 1-smooth.
 Since also $\Res_A(\calG)$ is 1-smooth, Remark~\ref{rem:unique} implies that $\theta\vert_A=\mathbf{1}$, and hence $A\subseteq\Ker(\theta)$.
 
Now, for arbitrary elements $x\in G$ and $y\in A$, put $z=[x,y]$.
 Since $A$ is normal in $G$, one has $z\in A$, and since $A$ is abelian, one has $[z,y]=1$.
 Then Proposition~\ref{prop:prop} applied to the subgroup of $G$ generated by $\{x,y\}$ yields $xyx^{-1}=x^{\theta(x)}$, and this completes the proof of statement~(i).

In order to prove statement~(ii), suppose that $y^p\in Z(\calG)$ for some $y\in G$.
Then $y^p\in\Ker(\theta)$, and since $\Img(\theta)$ has no non-trivial torsion, also $y$ lies in $\Ker(\theta)$.
Since $G$ is torsion-free by Remark~\ref{rem:torfree}, $y^p\neq1$.
Let $H$ be the subgroup of $G$ generated by $y$ and $x$, for some $x\in G$ such that $d(H)\geq2$.
Since $xy^px^{-1}=(y^p)^{\theta(x)}$, commutator calculus yields
\begin{equation}\label{eq:proof isolated}
 y^{p(1-\theta(x)^{-1})}=[x,y^p]=[x,y]\cdot[x,y]^y\cdots[x,y]^{y^{p-1}}.
\end{equation}
Put $z=[x,y]$, and let $S$ be the subgroup of $H$ generated by $y,z$.
Clearly, $\Res_S(\calG)$ is 1-smooth, and since $y,z\in\Ker(\theta)$, one has $\theta\vert_S=\mathbf{1}$, and thus $S/S'$ is a free abelian pro-$p$ group by Remark~\ref{rem:abs torfree}.
From \eqref{eq:proof isolated} one deduces
\begin{equation}\label{eq:proof isolated 2}
 y^{p(1-\theta(x)^{-1})}\cdot z^{-p}\equiv \left(y^{1-\theta(x)^{-1}}\cdot z^{-1}\right)^p\equiv 1\mod S'.
\end{equation}
Since $S/S'$ is torsion-free, \eqref{eq:proof isolated 2} implies that $z\equiv y^{1-\theta(x)^{-1}}\bmod\Phi(S)$, so that $S$ is generated by $y$, and $S\simeq\Z_p$, as $G$ is torsion-free.
Therefore, $S'=\{1\}$, and \eqref{eq:proof isolated 2} yields $[x,y]=y^{1-\theta(x)^{-1}}$, and this completes the proof of statement~(ii).
 \end{proof}
 
 \begin{rem}\rm
Let $G$ be a pro-$p$ group isomorphic to $\Z_p$, and let $\theta\colon G\to1+p\Z_p$ be a non-trivial orientation.
Then by Example~\ref{ex:1-smooth}--(a), $\calG=(G,\theta)$ is 1-smooth.
Since $G$ is abelian and $\theta(x)\neq1$ for every $x\in G$, $x\neq1$, $Z(\calG)=\{1\}$, still every subgroup of $G$ is normal and abelian.
 \end{rem}

 In view of the splitting of \eqref{eq:transition} (and in view of Remark~\ref{rem:split}), it seems natural to ask the following question.
 
 \begin{question}\label{ques:split}
 Let $\calG=(G,\theta)$ be a torsion-free 1-smooth pro-$p$ pair, with $d(G)\geq2$. 
 Is the pro-$p$ pair $\calG/Z(\calG)=(G/Z(\calG),\bar{\theta})$ 1-smooth?
 Does the short exact sequence of pro-$p$ groups
 \[
  \xymatrix{ \{1\}\ar[r] & Z(\calG)\ar[r] &G\ar[r] &G/Z(\calG)\ar[r] &\{1\}}
 \]
split?
 \end{question}

If $\calG=(G,\theta)$ is a torsion-free pro-$p$ pair, then either $\Ker(\theta)=G$, or $\Img(\theta)\simeq\Z_p$, hence in the former case one has $G\simeq\Ker(\theta)\rtimes (G/\Ker(\theta))$, as the right-side factor is isomorphic to $\Z_p$, and thus $p$-projective (cf. \cite[Ch.~III, \S~5]{nsw:cohn}).
Since $Z(\calG)\subseteq Z(\Ker(\theta))$ (and $Z(\calG)= Z(G)$ if $\Ker(\theta)=G$), and since $\Ker(\theta)$ is absolutely torsion-free if $\calG$ is 1-smooth, Question~\ref{ques:split} is equivalent to the following question (of its own group-theoretic interest):
if $G$ is an absolutely torsion-free pro-$p$ group, does $G$ split as direct product $$G\simeq Z(G)\times (G/Z(G))\:\text{?}$$
One has the following partial answer (cf. \cite[Prop.~5]{wurf}): if $G$ is absolutely torsion-free, and $Z(G)$ is finitely generated, then $\Phi_n(G)=Z(\Phi_n(G))\times H$, for some $n\geq1 $ and some subgroup $H\subseteq \Phi_n(G)$ (here $\Phi_n(G)$ denotes the iterated Frattini series of $G$, i.e., $\Phi_1(G)=G$ and $\Phi_{n+1}(G)=\Phi(\Phi_n(G))$ for $n\geq1$).

%%%%%%%%%%%%%%%%%%%%%%%%%%%%%%%%%%%%%%%%%%%%%%%%%%%%%%%55
%%%%%%%%%%%555
%%%%%%%%%%%%%%%%%%%%%%%%%%%%%%%%%%%%%%%%%%%%%%%%%%%%%%%%%5

\section{Solvable pro-$p$ groups}\label{sec:solvable}

%%%%%%%%%%%%%%%%%%%%%%%%%%%%%%%%%%%%%%%%%%%%%%%%%%%%

\subsection{Solvable pro-$p$ groups and maximal pro-$p$ Galois groups}

Recall that a (pro-$p$) group $G$ is said to be meta-abelian if there is a short exact sequence 
\[\xymatrix{ \{1\}\ar[r] & N\ar[r] & G\ar[r] & \bar G\ar[r]& \{1\}}
 \]
such that both $N$ and $\bar G$ are abelian; or, equivalently, if the commutator subgroup $G'$ is abelian. 
Moreover, a pro-$p$ group $G$ is solvable if the derived series $(G^{(n)})_{n\geq1}$ of $G$ --- i.e., $G^{(1)}=G$ and $G^{(n+1)}=[G^{(n)},G^{(n)}]$ --- is finite, namely $G^{(N+1)}=\{1\}$ for some finite $N$.

\begin{exam}\label{ex:solv loc unif}\rm
 A non-abelian locally uniformly powerful pro-$p$ group $G$ is meta-abelian: if $\theta\colon G\to1+p\Z_p$ is the associated orientation, then $G'\subseteq \Ker(\theta)^p$, and thus $G'$ is abelian.
\end{exam}

In Galois theory one has the following result by R.~Ware (cf. \cite[Thm.~3]{ware}, see also \cite{koe:solvable} and \cite[Thm.~4.6]{cq:bk}).

\begin{thm}\label{thm:ware}
 Let $\K$ be a field containing a root of 1 of order $p$ (and also $\sqrt{-1}$ if $p=2$).
 If the maximal pro-$p$ Galois group $G_{\K}(p)$ is solvable, then $\calG_{\K}$ is $\theta_{\K}$-abelian.
\end{thm}

%%%%%%%%%%%%%%%%%%%%%%%%%%%%%%%%%%%%%%%%%%%%%%55

\subsection{Proof of Theorem~\ref{thmA:intro} and Corollary~\ref{cor:intro}}

In order to prove Theorem~\ref{thmA:intro}, we prove first the following intermediate results
--- a consequence of W\"urfel's result \cite[Prop.~2]{wurf} and of \cite[Prop.~6.11]{cq:1smooth} ---, which may be seen as the ``1-smooth analogue'' of \cite[Thm.~2]{ware}.

\begin{prop}\label{thm:metabelian}
Let $\calG=(G,\theta)$ be a torsion-free 1-smooth pro-$p$ pair.
If $G$ is meta-abelian, then $\calG$ is $\theta$-abelian.
\end{prop}

\begin{proof}
Assume first that $\theta\equiv\mathbf{1}$ --- i.e., $G$ is absolutely torsion-free (cf. Remark~\ref{rem:abs torfree}).
Then $G$ is a free abelian pro-$p$ group by \cite[Prop.~2]{wurf}.

Assume now that $\theta\not\equiv\mathbf{1}$.
Since $\calG$ is 1-smooth, also $\Res_{\Ker(\theta)}(\calG)$ and $\Res_{\Ker(\theta)'}(\calG)$ are 1-smooth pro-$p$ pairs, and thus $\Ker(\theta)$ and $\Ker(\theta)'$ are absolutely torsion-free.
Moreover, $\Ker(\theta)'\subseteq G'$, and since the latter is abelian, also $\Ker(\theta)'$ is abelian, i.e., $\Ker(\theta)$ is meta-abelian.
Thus $\Ker(\theta)$ is a free abelian pro-$p$ group by \cite[Prop.~2]{wurf}.
Consequently, for arbitrary $y\in \Ker(\theta)$ and $x\in G$, the commutator $[x,y]$ lies in $\Ker(\theta)$ and $[[x,y],y]=1$.
Therefore, Proposition~\ref{prop:prop} implies that $xyx^{-1}=y^{\theta(y)}$ for every $x\in G$ and $y\in \Ker(\theta)$, 
namely, $\calG$ is $\theta$-abelian.
\end{proof} \medskip

Note that Proposition~\ref{thm:metabelian} generalizes \cite[Prop.~2]{wurf} from absolutely torsion-free pro-$p$ groups to 1-smooth pro-$p$ groups.
From Proposition~\ref{thm:metabelian}, we may deduce Theorem~\ref{thmA:intro}.

\begin{prop}\label{cor:solvable}
 Let $\calG=(G,\theta)$ be a torsion-free 1-smooth pro-$p$ pair.
If $G$ is solvable, then $G$ is locally uniformly powerful.
\end{prop}

\begin{proof}
Let $N$ be the positive integer such that $G^{(N)}\neq\{1\}$ and $G^{(N+1)}=\{1\}$.
Then for every $1\leq n\leq N$, the pro-$p$ pair $\Res_{G^{n}}(\calG)$ is 1-smooth, and $G^{(n)}$ is solvable, and moreover
$\theta\vert_{G^{(n)}}\equiv\mathbf{1}$ if $n\geq2$.

Suppose that $N\geq3$. 
Since $G^{(N-1)}$ is metabelian and $\theta\vert_{G^{(N-1)}}\equiv\mathbf{1}$, Proposition~\ref{thm:metabelian} implies that $G^{(N-1)}$ is a free abelian pro-$p$ group, and therefore $G^{(N)}=\{1\}$, a contradiction.
Thus, $N\leq2$, and $G$ is meta-abelian.
Therefore, Proposition~\ref{thm:metabelian} implies that the pro-$p$ pair $\calG$ is $\theta$-abelian, and hence $G$ is locally uniformly powerful (cf. \S~\ref{ssec:powerful}).
\end{proof} 

Proposition~\ref{cor:solvable} may be seen as the 1-smooth analogue of Ware's Theorem~\ref{thm:ware}.
Corollary~\ref{cor:intro} follows from Proposition~\ref{cor:solvable} and from the fact that a locally uniformly powerful pro-$p$ group is Bloch-Kato (cf. \S~\ref{ssec:powerful}).

\begin{cor}
 Let $\calG=(G,\theta)$ be a torsion-free 1-smooth pro-$p$ pair.
If $G$ is solvable, then $G$ is Bloch-Kato.
\end{cor}

This settles the Smoothness Conjecture for the class of solvable pro-$p$ groups.

%%%%%%%%%%%%%%%%%%%%%%%%%%%%%%%%%%%%%%%%%%%%%%%%%%%%%%%%%%5

\subsection{A Tits' alternative for 1-smooth pro-$p$ groups}

For maximal pro-$p$ Galois groups of fields one has the following Tits' alternative (cf. \cite[Cor.~1]{ware}).

\begin{thm}\label{thm:titsalt}
  Let $\K$ be a field containing a root of 1 of order $p$ (and also $\sqrt{-1}$ if $p=2$).
Then either $\calG_{\K}$ is $\theta_{\K}$-abelian, or $G_{\K}(p)$ contains a closed non-abelian free pro-$p$ group.
\end{thm}

Actually, the above Tits' alternative holds also for the class of Bloch-Kato pro-$p$ groups, with $p$ odd: if a Bloch-Kato pro-$p$ group $G$ does not contain any free non-abelian subgroups, then it can complete into a $\theta$-abelian pro-$p$ pair $\calG=(G,\theta)$ (cf. \cite[Thm.~B]{cq:bk}, this Tits' alternative holds also for $p=2$ under the further assumption that the Bockstein morphism $\beta\colon H^1(G,\Z/2)\to H^2(G,\Z/2)$ is trivial, see \cite[Thm.~4.11]{cq:bk}).

Clearly, a solvable pro-$p$ group contains no free non-abelian subgroups.

A pro-$p$ group is {\sl $p$-adic analytic} if it is a $p$-adic analytic manifold and
the map $(x,y)\mapsto x^{-1} y$ is analytic, or, equivalently, if it contains an open uniformly powerful subgroup (cf. \cite[Thm.~8.32]{ddsms}) --- e.g., the Heisenberg pro-$p$ group is analytic.
Similarly to solvable pro-$p$ groups, a $p$-adic analytic pro-$p$ group does not contain a free non-abelian subgroup (cf. \cite[Cor.~8.34]{ddsms}).

Even if there are several $p$-adic analytic pro-$p$ groups which are solvable (e.g., finitely generated locally uniformly powerful pro-$p$ groups), none of these two classes of pro-$p$ groups contains the other one: e.g.,
\begin{itemize}
 \item[(a)] the wreath product $\Z_p\wr\Z_p\simeq \Z_p^{\Z_p}\rtimes\Z_p$ is a meta-abelian pro-$p$ group,
 but it is not $p$-adic analytic (cf. \cite{shalev:wreath}); 
 \item[(b)] if $G$ is a pro-$p$-Sylow subgroup of $\mathrm{SL}_2(\Z_p)$, then $G$ is a $p$-adic analytic pro-$p$ group, but it is not solvable.
\end{itemize}

In addition, it is well-known that also for the class of pro-$p$ completions of right-angled Artin pro-$p$ groups one has a Tits' alternative: the pro-$p$ completion of a right-angled Artin pro-$p$ group contains a free non-abelian subgroup unless it is a free abelian pro-$p$ group (i.e., unless the associated graph is complete) --- and thus it is locally uniformly powerful.

In \cite{cq:1smooth}, it is shown that analytic pro-$p$ groups which may complete into a 1-smooth pro-$p$ pair are locally uniformly powerful.
Therefore, after the results in \cite{cq:1smooth} and \cite{SZ:RAAGs}, and Theorem~\ref{thmA:intro}, it is natural to ask whether a Tits' alternative, analogous to Theorem~\ref{thm:titsalt} (and its generalization to Bloch-Kato pro-$p$ groups), holds also for all torsion-free 1-smooth pro-$p$ pairs.

\begin{question}\label{conj:tits sec}
 Let $\calG=(G,\theta)$ be a torsion-free 1-smooth pro-$p$ pair, and suppose that $\calG$ is not $\theta$-abelian.
 Does $G$ contain a closed non-abelian free pro-$p$ group?
\end{question}

In other words, we are asking whether there exists torsion-free 1-smooth pro-$p$ pairs $\calG=(G,\theta)$ such that $G$ is not analytic nor solvable, and yet it contains no free non-abelian subgroups.
In view of Theorem~\ref{thm:titsalt} and of the Tits' alternative for Bloch-Kato pro-$p$ groups \cite[Thm.~B]{cq:bk}, a positive answer to Question~\ref{conj:tits sec} would corroborate the Smoothness Conjecture.

Observe that --- analogously to Quesion~\ref{ques:split} --- Question~\ref{conj:tits sec} is equivalent to asking whether an absolutely torsion-free pro-$p$ group which is not abelian contains a closed non-abelian free subgroup.
Indeed, by Proposition~\ref{prop:prop} (in fact, just by \cite[Prop.~5.6]{cq:1smooth}), if $\calG=(G,\theta)$ is a torsion-free 1-smooth pro-$p$ pair and $\Ker(\theta)$ is abelian, then $\calG$ is $\theta$-abelian.

%%%%%%%%%%%%%%%%%%%%%%%%%%%%%%%%%%%%5

%%%%%%%%%%%%%%%%%%%%%%%%%%%%%%%%%%%%%%%%%%%%%%%%%%%%%%%55
%%%%%%%%%%%555
%%%%%%%%%%%%%%%%%%%%%%%%%%%%%%%%%%%%%%%%%%%%%%%%%%%%%%%%%5

\subsection*{Acknowledgement}
{\small 
The autor wishes to thank I.~Efrat, J.~Min\'a\v c, N.D. T\^an and Th.~Weigel for working together on maximal pro-$p$ Galois groups and their cohomology; and P.~Guillot and I.~Snopce for the discussions on 1-smooth pro-$p$ groups.
Also, the author wishes to thank the editors of CMB-BMC, for their
helpfulness, and the anonymous referee.
}

\end{document}